\documentclass[10pt,letterpaper]{article}

\usepackage{latexsym}
\usepackage{graphics}
\usepackage{amsmath}
\usepackage{amsthm}
\usepackage{amsopn}
\usepackage{mathtools}
\usepackage{amstext}
\usepackage{amscd}
\usepackage{amsfonts}
\usepackage{amssymb}
\usepackage{authblk}

\usepackage{dsfont} 
\usepackage{comment}
\usepackage[active]{srcltx}
\usepackage{graphicx,color}
\usepackage[abs]{overpic}

\newtheorem{theorem}{Theorem}

\newtheorem{problem}{Problem}

\newcommand{\re}{\mathds{R}}
\newcommand{\du}{\mathds{D}}

\newcommand{\co}{\mathds{C}}
\newcommand{\na}{\mathds{N}}

\newcommand{\dps}{\displaystyle}

\newcommand{\dsty}{\displaystyle}

\title{On a class of biorthogonal polynomials on the unit circle}

\author[1]{J. Borrego--Morell}
\author[2]{ Fernando Rodrigo Rafaeli }

\affil[1]{ Departamento de Matem\'atica Aplicada, 
 Universidade Estadual Paulista--UNESP, Campus S\~ao Jos\'e do Rio Preto\\e--mail: jbmorell@gmail.com}
\affil[2]{
 Faculdade de Matem\'atica, Universidade Federal de Uberl\^andia--UFU\\
 e--mail: rafaeli@ufu.br}
\begin{document}
\maketitle
\begin{abstract}

A system of biorthogonal  polynomials with respect to a complex valued measure supported on the unit circle  is considered   and   all the terms with bounds are explicitly given for the remainder   of  an asymptotic formula  given by R. Askey for this  system. An electrostatic interpretation   for the zeros of a class of para-orthogonal polynomials associated with the biorthogonal system is also considered.
\end{abstract}

\section{Introduction and statement of the results}

R. Askey in \cite{A82} introduced the following system $\{P_n,Q_n\}$ of  polynomials 

\begin{eqnarray}\label{Pn}
P_n(z; \alpha, \beta)&=& {_2F_1}(-n, \alpha+\beta+1; 2 \alpha+1; 1-z)\\
\nonumber Q_n(z; \alpha, \beta)&=&P_n(z; \alpha, -\beta),
\end{eqnarray}
which is biorthogonal with respect to the complex valued weight $\omega(\theta)=(1-e^{\imath \theta})^{\alpha+\beta}(1-e^{-\imath \theta})^{\alpha-\beta}=(2-2\cos\theta)^{\alpha}(-e^{\imath\theta})^{\beta}, \theta\in [-\pi,\pi], \Re(\alpha)>-\frac{1}{2}$, that is
\begin{equation}\label{PnRo}
\frac{1}{2\pi}\int_{-\pi}^{\pi}P_n(e^{\imath \theta};\alpha, \beta)Q_m(e^{-\imath \theta};\alpha, \beta)\omega(\theta)d\theta=
 \frac{\Gamma(2\alpha+1)}{\Gamma(\alpha+\beta+1)\Gamma(\alpha-\beta+1)}\frac{n!}{(2\alpha+1)_n}\delta_{n,m},
\end{equation}
where $\Gamma$ denotes the Euler Gamma function. The biorthogonality \eqref{PnRo} was stated in \cite{A82} in a slightly different form and a formal proof was given later in \cite[pp. 16--17]{A85}. Other proofs of the biorthogonality  have been given by several authors, please see \cite{N86} for some historical considerations. 

In \cite[pp. 17]{A85} Askey obtains the formula (also obtained previously by Basor in \cite{B78}  for a more general class of weights) 

\begin{eqnarray}\label{ask0}
P_n(e^{\frac{\imath \theta}{n}};\alpha,\beta)\sim {}_{1}F_{1}(\alpha+\beta+1;2\alpha+1;\imath\theta),\quad \mbox{as} \quad n\rightarrow \infty,
\end{eqnarray}
which is analogous to the one for Jacobi polynomials $P_n^{(\alpha,\beta)}$

$$n^{-\alpha}P_n^{(\alpha,\beta)}(\cos \theta/n)\sim \left(\frac{\theta}{2}\right)^{-\alpha}J_{\alpha}(\theta), \quad \mbox{as} \quad n\rightarrow \infty.$$

Askey remarked that it is interesting to understand  the effect of the zeros of the weight function on the asymptotic behavior of the orthogonal polynomials and this raises the  question, which will be referred to in the present manuscript as  Askey's problem, of how to obtain the first term or preferably, more terms for the remainder in the asymptotic formula \eqref{ask0} as well as bounds for the remainder.

 Temme in \cite{N86}  proved that, for $z$ and $(\alpha,\beta)$ varying in compact subsets of $\co \setminus \{0\}$ and $\Omega=\{(\alpha,\beta)\in \co^2:\Re(\alpha+\beta)>-1, \Re(\alpha-\beta)\geq 0\}$ respectively, it holds that
\begin{multline*}
P_n(z; \alpha, \beta)\sim\frac{\Gamma(2 \alpha +1)}{\Gamma(\alpha+\beta+1)} z^{\alpha-\beta-1} \left( \frac{\ln z}{z-1} \right)^{2 \alpha}\times\\
\left(\varphi_0 \sum_{k=0}^{\infty} \frac{A_k}{(n+1)^k}+\varphi_1 \sum_{k=0}^{\infty} \frac{B_k}{(n+1)^k}+R_p\right), \quad \mbox{as}\, n\rightarrow \infty ,
\end{multline*}
where $\varphi_0=\frac{\Gamma(\alpha+\beta+1)}{\Gamma(2\alpha+1)}{_1F_1}(\alpha+\beta+1,2\alpha+1;(n+1) \ln z)$, $\varphi_1=\frac{\Gamma(\alpha+\beta+2)}{\Gamma(2\alpha+2)}{_1F_1}(\alpha+\beta+2,2\alpha+2;(n+1) \ln z)$   and $A_k$, $B_k$ are coefficients defined by the recursion relations \cite[(2.13)]{N86}. Moreover, a bound for the  remainder $R_p$ for this asymptotic expansion defined as 

\begin{multline*}
P_n(z; \alpha, \beta)=\frac{\Gamma(2 \alpha +1)}{\Gamma(\alpha+\beta+1)} z^{\alpha-\beta-1} \left( \frac{\ln z}{z-1} \right)^{2 \alpha}\times\\
\left(\varphi_0 \sum_{k=0}^{p-1} \frac{A_k}{(n+1)^k}+\varphi_1 \sum_{k=0}^{p-1} \frac{B_k}{(n+1)^k}+R_p\right), \,  n, p\in \na,
\end{multline*}
is given by
$$
|R_p|\leq \frac{M_p}{(n+1)^p}  \left| \frac{\Gamma(\alpha+\beta+1)}{\Gamma(2 \alpha+1)}\right| \left| {_1F_1}(\alpha+\beta+1; 2 \alpha+1; (n+1) \Re \ln z)\right|,
$$
where $M_p$ is some positive constant depending only on $p$. Temme remarked that the evaluation of the  coefficients $A_k$, $B_k$ is  difficult, especially near or at unity. This   asymptotic expansion gives, for $z=e^{\frac{\imath \theta}{n}}$ as a particular case,  an  answer to Askey's  problem, provided that $(\alpha,\beta)$ varies in compact subsets of $\Omega$.

In the present manuscript  an  explicit  expression of all  the terms of the asymptotic formula \eqref{ask0} is presented.  Bounds for the remainder for this expansion, which turns out to be convergent are also considered.

Denote by $\du$   the open  unit disk $\{z:|z|< 1\}$.  For $z,\alpha \in \co$, we choose  $-\pi<\arg z\leq\pi$. We define the functions $z^\alpha$ and $\ln z$ according to the branch of $\arg z$.

 $B_{n}^{(\alpha)}(x)$  the generalized Bernoulli polynomials are defined  using  the generating function \cite[Sec. 2.8]{Lu}
$$
\left( \frac{z}{e^z-1} \right)^{\alpha}e^{xz}=\displaystyle\sum_{n=0}^{\infty}B_{n}^{(\alpha)}(x)\frac{z^n}{n!}, \ |z|<2\pi;\ x,\alpha\in \co.
$$

The asymptotic expansion reads as:

\begin{theorem} \label{Th2}
Assume that $(\alpha,\beta)\in\Omega$, then
\begin{multline*}
 P_n\left(e^{\frac{\imath\theta}{n}};\alpha,\beta \right)= {}_{1}F_1(\alpha+\beta+1;2\alpha+1;\imath \theta)+\\
 \sum_{j=1}^{k}\sum_{|i|=j}
\frac{B_{i_1}^{(-\alpha-\beta)}(\alpha-\beta)}{i_1!}\frac{B_{i_2}^{(-\alpha+\beta+1) }(0)}{i_2!}\frac{B_{i_3}^{(2\alpha)}(0)}{i_3!}\times \\ \frac{(\alpha+\beta+1)_{i_1}(\alpha-\beta)_{i_2}}{(2\alpha+1)_{i_1+i_2}}{}_1F_1(1+\alpha+\beta+i_1;1+2\alpha+i_1+i_2;\imath \theta)\left(\frac{\imath \theta}{n}\right)^{j}
 + R_{k,n}(\theta),  \\
   \theta\in[-\pi,\pi), \quad n\in \na,
\end{multline*}
where $|i|=i_1+i_2+i_3,\,  i_1,i_2,i_3\in \na\cup \{0\}$ and 
\begin{equation*}
 |R_{k,n}(\theta)|\leq  \frac{\Gamma(\Re(\alpha+\beta+1)\Gamma(\Re(\alpha-\beta))}{|\Gamma(\alpha+\beta+1)\Gamma(\alpha-\beta)|}\left|\frac{1}{1-\frac{2\theta}{3\pi n}}\right|\left|\frac{2\theta}{3\pi n}\right|^{k+1} C(\alpha,\beta),
 \end{equation*}
 
$$C(\alpha,\beta)=\max_{|z|=\frac{3\pi}{2}}\left|e^{z(\alpha-\beta)}\left(\frac{z}{e^{z}-1} \right)^{-\alpha-\beta} \right|\max_{|z|=\frac{3\pi}{2}}\left| \left(\frac{z}{e^{z}-1} \right)^{-\alpha+\beta+1}\right|\max_{|z|=\frac{3\pi}{2}}\left|\left(\frac{z}{e^z-1} \right)^{2\alpha}\right|.$$

\end{theorem}

\subsubsection*{An electrostatic model for a class of para--orthogonal polynomials}

Asymptotic properties and electrostatic interpretation of the zeros of orthogonal polynomials,  are both commonly studied themes in the theory of  orthogonal polynomials and mathematical physics.  

Here it is also  shown  that  the zeros of a class of para--orthogonal polynomials associated to the biorthogonal system \eqref{Pn} have an  electrostatic interpretation in the unit circle  very much  in the classical sense of Stieltjes. Similar  electrostatic models exist for the zeros  of other families of orthogonal polynomials with respect to measures supported on the real line. Gr\"unbaum in \cite{gru98} described an electrostatic model for the zeros of the Koornwinder polynomials, and Ismail in \cite{ismail00} gave another model for the zeros of orthogonal polynomials with respect to a measure satisfying certain integrability conditions with an absolutely continuous part  and a finite discrete part.

 It is well known that, for a positive definite functional, the zeros of the Szeg\H{o} polynomials all lie in $\du$. In order to develop quadrature rules on the unit circle, it is useful to have orthogonal polynomials   with respect to a linear functional whose zeros lie on $\partial \du $. Motivated by this fact Jones,  Nj{\aa}stad, and  Thron in  \cite[pp. 130]{JNT} defined a sequence $\{X_n\}_{n=0}^{\infty}$ of para--orthogonal polynomials with respect to a quasi-definite linear functional $\mu$, if for each $n\geq0$, $X_n$ is a polynomial of degree $n$ satisfying
\begin{equation*}
\langle X_n, 1\rangle \neq 0,\ \ \ \langle X_n, z^m\rangle = 0\ \ \mbox{for}\ \ 1\leq m \leq n-1,\ \ \mbox{and}\ \ \langle X_n, z^n\rangle \neq 0,
\end{equation*}
where $\langle X,Y\rangle=\mu(X(z)\overline{Y}(1/z)); X,Y\in \Lambda$, $\Lambda $ being the space of all Laurent polynomials.
According to  these authors, if $\Phi_{n}$ is the $n$--th monic polynomial with respect to a linear functional $\mu$, the polynomial
$$
B_n(z;c)=\Phi_n(z)+c\Phi_n^{*}(z), \ |c|=1,
$$
where $\Phi_n^{*}$ is the reciprocal polynomial, is para-orthogonal polynomial of degree $n$. From \cite[Th. 6.2]{JNT}, if $\mu$ is a positive definite functional, the $n$ zeros of the para-orthogonal polynomials are simple and lie on  $\partial \du $.

The zeros of a class of para--orthogonal polynomials described in \cite{ranga} associated with a  positive definite functional defined using the  weight function of the biorthogonal system \eqref{Pn} obey an electrostatic model.

Consider the moment functional

$$\mu(X)=\frac{|\Gamma(\alpha+\beta+1)|^{2}}{2\pi\Gamma(2\alpha+1)}\int_{-\pi}^{\pi}X(e^{\imath \theta})\omega(\theta)d\theta. $$
Notice that $\mu$ is positive definite if and only if the weight $\omega$ is positive. From the expression for $\omega$, this happens when  $\alpha\in \re, \alpha>-\frac{1}{2}$ and $\imath\beta\in \re$.  It will be  assumed in this section that $\alpha$ and $\beta$ satisfy these conditions. Notice that for this case  $\omega(\theta)=2^{2\alpha}e^{(\pi-\theta)\Im(\beta)}\sin^{2\alpha}(\frac{\theta}{2})$.

Ranga, in \cite{ranga}  studied the sequence $\dsty \left\{\frac{(2\alpha+1)_{n}}{(\alpha+\beta+1)_{n}}P_{n}(z;\alpha,\beta)\right\}_{n=0}^{\infty}$ of monic  orthogonal polynomials with respect to $\mu$, and the  author obtained that   the polynomial

$$B_n\left(z;\frac{(\alpha-\beta)_{n+1}}{(\alpha+\beta)_{n+1}}\right)=\frac{(2\alpha)_n}{(\alpha+\beta)_n}\, {_2F_1}(-n, \alpha+\beta; 2\alpha; 1-z), \ \alpha\neq 0, $$
is the $n$--th  para--orthogonal monic polynomial  with respect to the positive  definite linear functional $\mu$.

The electrostatic model for the zeros of $B_{n}$ can be formulated as the solution of the following  problem:

\begin{problem}
Let $p,q$ be two given  real numbers, $p\neq 0$. If $n$ unit masses, $n\geq 2$ at the variable points $\left\{e^{\imath \theta_1},\ldots, e^{\imath \theta_n}\right\}$ in the unit circumference, and one \emph{fixed} mass point $p$ at $+1$ is  considered, for what position of the points $\left\{e^{\imath \theta_1},\ldots, e^{\imath \theta_n}\right\}$ does the expression

\begin{equation}\label{E}
E(\theta_1,\ldots,\theta_n)=\\
 \sum_{k\neq j}\ln\frac{1}{|e^{\imath \theta_k}-e^{\imath \theta_j}|}+
 p\sum_{ j=1}^{n}\ln\frac{1}{|1-e^{\imath\theta_j}|}+q\sum_{j=1}^{n}\theta_j,\
 \theta_j\in (0,2\pi),
\end{equation}
become a minimum?
\end{problem}

The solution of the above problem is given in the following theorem:

\begin{theorem}\label{Th3}
Let $p,q$ be two real numbers, $p\neq 0$ and let $\{\theta_1^{*},\ldots,\theta_n^{*}\}, \theta_1^{*}<\cdots <\theta^{*}_n$ be a system of values such that  $\{e^{\imath\theta^{*}_1},\ldots,e^{\imath\theta^{*}_n}\}$  are the zeros of the para-orthogonal polynomial $B_n$ with parameters $\alpha=p$ and $\beta=2\imath q$, then $\nabla_{\theta}E(\theta^{*}_1,\ldots,\theta^{*}_n)=0$. Moreover,   $E$ attains its global minimum at the point $(\theta^{*}_1,\ldots,\theta^{*}_n)$  if and only if $p>0$.  
\end{theorem}



\section{ Askey's problem }\label{PAsk}

For the proof of Theorem \ref{Th2},   the  following integral representation for $P_n(z;\alpha,\beta)$, with $ (\alpha,\beta)\in\Omega$ and  $z\in\co\setminus \{0\}$ is used, which is straightforward from \cite[Sect. 2 Eqs. (2.1)--(2.4)]{N86},

\begin{equation}\label{nico}
P_{n}(z;\alpha,\beta)=
\frac{\Gamma(2\alpha+1)}{\Gamma(\alpha+\beta+1)\Gamma(\alpha-\beta)} 
\int_{0}^{1}\left(\frac{z^u-1}{z-1} \right)^{\alpha+\beta}\left(\frac{z-z^u}{z-1} \right)^{\alpha-\beta-1}\frac{e^{n u \ln z}z^u}{z-1}\ln z du.
\end{equation}

The proof of  the theorem is now given.

\begin{proof}(Of Theorem \ref{Th2})

From \eqref{nico}  

\begin{multline}\label{inttheta}
P_{n}(e^{\imath \frac{\theta}{n}};\alpha,\beta)=\frac{\Gamma(2\alpha+1)}{\Gamma(\alpha+\beta+1)\Gamma(\alpha-\beta)} \times \\
\dsty \int_{0}^{1}\left(\frac{e^{\frac{\imath \theta u}{n}}-1}{e^{\frac{\imath \theta}{n}}-1} \right)^{\alpha+\beta}\left(\frac{e^{  \frac{\imath \theta}{n}}-e^{  \frac{\imath \theta u}{n}}}{e^{ \frac{\imath\theta}{n}}-1} \right)^{\alpha-\beta-1}e^{\frac{\imath u\theta}{n}}\frac{\frac{\imath\theta}{n}}{e^{ \frac{\imath\theta}{n}}-1} e^{\imath u\theta} du.
\end{multline}

Notice that if $0\leq x,y \leq\pi$ then
\begin{equation}\label{eqa}
0\leq \arg\left(\dps\frac{e^{\imath x}-1}{\imath x}\right)\leq\dps\frac{\pi}{2}, \   -\dps\frac{\pi}{2}\leq \arg\left(\dps\frac{\imath y}{e^{\imath y}-1}\right)\leq  0.
\end{equation}

In a similar way, if $-\pi<x,y\leq 0$  then
\begin{equation}\label{eqb}
\dps-\frac{\pi}{2}<\arg\left( \dps\frac{e^{\imath x}-1}{ix}\right)\leq0, \  0\leq \arg\left( \dps\frac{\imath y}{e^{\imath y}-1}\right)<\dps\frac{\pi}{2}.
\end{equation}

Relations \eqref{eqa} and \eqref{eqb}  give, for $0\leq x,y \leq\pi$ or  $-\pi<x,y\leq 0$,
\begin{equation}\label{eq223}
\dps-\frac{\pi}{2}\leq \arg\left( \dps\frac{e^{\imath x}-1}{\imath x}\right)+\arg\left( \dps\frac{\imath y}{e^{\imath y}-1}\right)\leq \dps\frac{\pi}{2}.
\end{equation}

It is well know that, if $z_1,z_2,\gamma\in \co$, then
\begin{equation}\label{prodcom}(z_1z_2)^{\gamma}=e^{2k\pi\gamma\imath}z_1^{\gamma}z_2^{\gamma},
\end{equation}
where $\dsty k=\left \lfloor\frac{\arg(z_1z_2)-\arg z_1-\arg z_2}{2\pi}\right \rfloor$, here  $\lfloor a \rfloor$ stands for the floor function  of $a\in \re$. 

From \eqref{eq223} and \eqref{prodcom},  if $\dsty \theta_n=\frac{\imath \theta}{n}, \theta\in [-\pi,\pi), n\in \na$ and $u\in[0,1]$
\begin{equation}\label{int1}
\left(\frac{e^{\theta_n u}-1}{e^{\theta_n}-1} \right)^{\alpha+\beta}=u^{\alpha+\beta}\left( \frac{e^{\theta_n u}-1}{\theta_n u} \right)^{\alpha+\beta}\left( \frac{\theta_n}{e^{\theta_n}-1} \right)^{\alpha+\beta},
\end{equation}

\begin{equation}\label{int2}
\left(\frac{e^{\theta_n}-e^{\theta_n u}}{e^{ \theta_n}-1} \right)^{\alpha-\beta-1}=
(1-u)^{\alpha-\beta-1}e^{u\theta_n(\alpha-\beta-1)}
\left( \frac{e^{(1-u)\theta_n}-1}{ (1-u)\theta_n} \right)^{\alpha-\beta-1}\left( \frac{\theta_n}{e^{\theta_n}-1} \right)^{\alpha-\beta-1}.
\end{equation}

Substituting \eqref{int1} and \eqref{int2} in \eqref{inttheta}
\begin{multline}\label{intthetapre}
P_{n}(e^{\theta_n};\alpha,\beta)=\frac{\Gamma(2\alpha+1)}{\Gamma(\alpha+\beta+1)\Gamma(\alpha-\beta)} \int_0^1 u^{\alpha+\beta}(1-u)^{\alpha-\beta-1}\times \\
e^{u\theta_n(\alpha-\beta)}\left(\frac{u\theta_n}{e^{u\theta_n}-1} \right)^{-\alpha-\beta}
\left(\frac{\theta_n(1-u)}{e^{\theta_n(1-u)}-1} \right)^{-\alpha+\beta+1}\left(\frac{\theta_n}{e^\theta_n-1} \right)^{2\alpha} e^{n\theta_n u}du.
\end{multline}

From the generating functions for the generalized Bernoulli  polynomials \cite[Sec. 2.8]{Lu},  
\begin{eqnarray*}
e^{u\theta_n(\alpha-\beta)}\left(\frac{u\theta_n}{e^{u\theta_n}-1} \right)^{-\alpha-\beta}&=&\sum_{j=0}^{\infty}B_{j}^{(-\alpha-\beta)}(\alpha-\beta)\frac{(u\theta_n)^j }{j!},\\
\left(\frac{\theta_n(1-u)}{e^{\theta_n(1-u)}-1} \right)^{-\alpha+\beta+1}&=&\sum_{j=0}^{\infty}B_{j}^{(-\alpha+\beta+1)}(0)\frac{(1-u)^{j}\theta_n^j}{j!},\\
\left(\frac{\theta_n}{e^\theta_n-1} \right)^{2\alpha}&=&\sum_{j=0}^{\infty}B_{j}^{(2\alpha)}(0)\frac{\theta_n^j}{ j!} .
\end{eqnarray*}

Substituting these relations in \eqref{intthetapre}
\begin{multline}\label{forrem}
P_{n}(e^{\theta_n};\alpha,\beta)=\\
\frac{\Gamma(2\alpha+1)}{\Gamma(\alpha+\beta+1)\Gamma(\alpha-\beta)} \int_{0}^{1}u^{\alpha+\beta}(1-u)^{\alpha-\beta-1}
\sum_{j=0}^{\infty}B_{j}^{(-\alpha-\beta)}(\alpha-\beta)\frac{u^j\theta_n^j}{j!}\times \\
\sum_{j=0}^{\infty}B_{j}^{(-\alpha+\beta+1)}(0)\frac{(1-u)^{j}\theta_n^j}{j!} \sum_{j=0}^{\infty}B_{j}^{(2\alpha)}(0)\frac{\theta_n^j}{j!} e^{n\theta_n u}du\\
=\frac{\Gamma(2\alpha+1)}{\Gamma(\alpha+\beta+1)\Gamma(\alpha-\beta)}\int_{0}^{1}u^{\alpha+\beta}(1-u)^{\alpha-\beta-1}
\sum_{j=0}^{\infty} b_j(u)\theta_n^{j}e^{n\theta_n u}du,
\end{multline}
where $\dsty b_j(u)=\sum_{\begin{subarray}{c}
i_1+i_2+i_3=j\\
i_1,i_2,i_3\in \na\cup \{0\} 
\end{subarray}}\frac{B_{i_{1}}^{(-\alpha-\beta)}(\alpha-\beta)}{i_{1}!} \frac{B_{i_{2}}^{(-\alpha+\beta+1)}(0)}{i_{2}!}\frac{B_{i_{3}}^{(2\alpha)}(0)}{i_{3}!}u^{i_{1}}(1-u)^{i_2} $ are the coefficients of the Taylor development about $v=0$ of the function
$$e^{uv(\alpha-\beta)}\left(\frac{uv}{e^{uv}-1} \right)^{-\alpha-\beta}
\left(\frac{v(1-u)}{e^{v(1-u)}-1} \right)^{-\alpha+\beta+1}\left(\frac{v}{e^v-1} \right)^{2\alpha}.$$

As the series of the last equality in \eqref{forrem} converges uniformly in $[0,1]$,  

\begin{multline}\label{preaskey}
P_n(e^{\imath \frac{\theta}{n}};\alpha,\beta)=
\frac{\Gamma(2\alpha+1)}{\Gamma(\alpha+\beta+1)\Gamma(\alpha-\beta)} \sum_{j=0}^{\infty}\left(\sum_{|i|=j}\frac{B_{i_{1}}^{(-\alpha-\beta)}(\alpha-\beta)}{i_{1}!}\frac{B_{i_{2}}^{(-\alpha+\beta+1)}(0)}{i_{2}!}\frac{B_{i_{3}}^{(2\alpha)}(0)}{i_{3}!}  \times \right. \\
\left. \int_{0}^{1}u^{\alpha+\beta+i_{1}}(1-u)^{\alpha-\beta-1+i_{2}} e^{\imath\theta u}du\right) \left(\frac{\imath \theta}{n}\right)^{j}=\\
\sum_{j=0}^{\infty}\sum_{|i|=j}\frac{B_{i_{1}}^{(-\alpha-\beta)}(\alpha-\beta)}{i_{1}!}\frac{B_{i_{2}}^{(-\alpha+\beta+1)}(0)}{i_{2}!} \frac{B_{i_{3}}^{(2\alpha)}(0)}{i_{3}!}\times \\
\frac{(\alpha+\beta+1)_{i_1}(\alpha-\beta)_{i_2}}{(2\alpha+1)_{i_1+i_2}}{_1}F_1(1+\alpha+\beta+i_1;1+ 2 \alpha+i_1+i_2; \imath \theta)\left(\frac{\imath \theta}{n}\right)^{j},
\end{multline}
where $|i|=i_1+i_2+i_3,\,  i_1,i_2,i_3\in \na\cup \{0\}$. Consider now the remainder  $R_{k,n}$ defined as

\begin{multline}\label{remAsk}
P_n(e^{\imath \frac{\theta}{n}};\alpha,\beta)=\sum_{j=0}^{k}\sum_{|i|=j}\frac{B_{i_{1}}^{(-\alpha-\beta)}(\alpha-\beta)}{i_{1}!}\frac{B_{i_{2}}^{(-\alpha+\beta+1)}(0)}{i_{2}!} \frac{B_{i_{3}}^{(2\alpha)}(0)}{i_{3}!}\times \\
\frac{(\alpha+\beta+1)_{i_1}(\alpha-\beta)_{i_2}}{(2\alpha+1)_{i_1+i_2}}{_1}F_1(1+\alpha+\beta+i_1; 1+2 \alpha+i_1+i_2; \imath \theta)\left(\frac{\imath \theta}{n}\right)^{j}+R_{k,n}(\theta).
\end{multline}

From \eqref{forrem}  

\begin{multline}\label{remAsk1}
P_n(e^{\imath \frac{\theta}{n}};\alpha,\beta)=\frac{\Gamma(2\alpha+1)}{\Gamma(\alpha+\beta+1)\Gamma(\alpha-\beta)}\times\\
\left(\sum_{j=0}^{k} +\sum_{j=k+1}^{\infty}\right)\left(\int_{0}^{1}u^{\alpha+\beta}(1-u)^{\alpha-\beta-1}
b_j(u) e^{\imath \theta u}du\right)\left(\frac{\imath \theta}{n}\right)^{j}.
\end{multline}
From \eqref{remAsk} and \eqref{remAsk1} it can be deduced

\begin{equation}\label{remAsk2}
R_{k,n}(\theta)=
\frac{\Gamma(2\alpha+1)}{\Gamma(\alpha+\beta+1)\Gamma(\alpha-\beta)}\sum_{j=k+1}^{\infty}\left(\int_{0}^{1}u^{\alpha+\beta}(1-u)^{\alpha-\beta-1}
b_j(u) e^{\imath\theta u}du\right)\left(\frac{\imath \theta}{n}\right)^{j}.
\end{equation}

From  Cauchy's estimate \cite[(25) pp. 122]{A66},  
\begin{eqnarray}\label{cauchyestim}
|b_j(u)|\leq M(r) r^{-j},\, 0<r<2\pi,
\end{eqnarray}
where
\begin{multline*}
M(r)=\max_{|v|=r}\left|e^{uv(\alpha-\beta)}\left(\frac{uv}{e^{uv}-1} \right)^{-\alpha-\beta} \left(\frac{v(1-u)}{e^{v(1-u)}-1} \right)^{-\alpha+\beta+1}\left(\frac{v}{e^v-1} \right)^{2\alpha}\right|\leq\\
    \max_{|v|=r}\left|e^{uv(\alpha-\beta)}\left(\frac{uv}{e^{uv}-1} \right)^{-\alpha-\beta} \right|\max_{|v|=r}\left| \left(\frac{v(1-u)}{e^{v(1-u)}-1} \right)^{-\alpha+\beta+1}\right|\max_{|v|=r}\left|\left(\frac{v}{e^v-1} \right)^{2\alpha}\right|.
\end{multline*}

The change of variables $z_1=v(1-u), z_2=uv$ and   the fact that $u\in [0,1]$ give,

\begin{eqnarray*}
M(r)\leq \max_{|z_2|=r}\left|e^{z_2(\alpha-\beta)}\left(\frac{z_2}{e^{z_2}-1} \right)^{-\alpha-\beta} \right|\max_{|z_1|=r}\left| \left(\frac{z_1}{e^{z_1}-1} \right)^{-\alpha+\beta+1}\right|\max_{|v|=r}\left|\left(\frac{v}{e^v-1} \right)^{2\alpha}\right|,
\end{eqnarray*}
  therefore, from \eqref{cauchyestim}
\begin{multline}\label{cauchyestimf}
 |b_j(u)|\leq \\
  \left(\frac{2}{3\pi}\right)^{j}\max_{|z|=\frac{3\pi}{2}}\left|e^{z(\alpha-\beta)}\left(\frac{z}{e^{z}-1} \right)^{-\alpha-\beta} \right|\max_{|z|=\frac{3\pi}{2}}\left| \left(\frac{z}{e^{z}-1} \right)^{-\alpha+\beta+1}\right|\max_{|z|=\frac{3\pi}{2}}\left|\left(\frac{z}{e^z-1} \right)^{2\alpha}\right|.
 \end{multline}

 From the expression for the remainder \eqref{remAsk2} and from \eqref{cauchyestimf}, one obtains
\begin{multline*}
|R_{k,n}(\theta)|\leq \frac{\Gamma(\Re(\alpha+\beta+1)\Gamma(\Re(\alpha-\beta))}{|\Gamma(\alpha+\beta+1)\Gamma(\alpha-\beta)|}C(\alpha,\beta)\sum_{j=k+1}^{\infty} \left|\frac{2\theta}{3\pi n}\right|^{j}\leq \\
\frac{\Gamma(\Re(\alpha+\beta+1)\Gamma(\Re(\alpha-\beta))}{|\Gamma(\alpha+\beta+1)\Gamma(\alpha-\beta)|}C(\alpha,\beta)\left|\frac{1}{1-\frac{2\theta}{3\pi n}}\right|\left|\frac{2\theta}{3\pi n}\right|^{k+1}, 
\end{multline*}
where 
$$C(\alpha,\beta)=\max_{|z|=\frac{3\pi}{2}}\left|e^{z(\alpha-\beta)}\left(\frac{z}{e^{z}-1} \right)^{-\alpha-\beta} \right|\max_{|z|=\frac{3\pi}{2}}\left| \left(\frac{z}{e^{z}-1} \right)^{-\alpha+\beta+1}\right|\max_{|z|=\frac{3\pi}{2}}\left|\left(\frac{z}{e^z-1} \right)^{2\alpha}\right|.$$
\end{proof}

\section{An electrostatic model for zeros of a class of para-orthogonal polynomials}\label{PEI}

\begin{proof}(Of Theorem \ref{Th3})

   From the relation

$$
\displaystyle\frac{\partial}{\partial\theta_j}\ln\frac{1}{|e^{\imath\theta_k}-e^{\imath\theta_j}|}=\Im\left(\frac{e^{\imath\theta_j}}{e^{\imath\theta_j}-e^{\imath\theta_k}} \right),
$$
it can be  deduced that the partial derivatives of $E$ can be expressed as

\begin{eqnarray*}
\displaystyle\frac{\partial E}{\partial\theta_j}=\sum_{k\neq j}\Im\left(\frac{e^{\imath\theta_j}}{e^{\imath\theta_j}-e^{\imath\theta_k}} \right)-
\Im\left(\left(\frac{p}{1-e^{\imath\theta_j}}+\frac{\frac{n+p-1}{2}-\imath q}{e^{\imath\theta_j}} \right)e^{\imath\theta_j} \right).
\end{eqnarray*}

Notice that  the  auxiliary term $\dsty \frac{n+p-1}{2}$ is introduced into the above expression. This term, as will be seen below, completes the expression for the differential equation that defines the para--orthogonal polynomial. 

By introducing the polynomial $\dsty f(z)=\prod_{j=1}^{n}\left(z-e^{\imath\theta_j}\right)$, a straightforward calculation shows that the equation  $\nabla_{\theta}E(\theta_1,\ldots,\theta_n)=0$ can be expressed as
\begin{multline*}
\displaystyle\frac{\partial E}{\partial\theta_j}=
\Im\left(z_j\frac{1}{2}\frac{f^{\prime\prime}(z_j)}{f^{\prime}(z_j)} -\left(\frac{p}{1-z_j}+
\frac{\frac{n+p-1}{2}-\imath q}{z_j}\right)z_j\right)=\\
\Im
\left(\frac{z_j(1-z_j)f^{\prime\prime}(z_j)-(n+p-1-2\imath q-(n-p-1-2\imath q)z_j)f^{\prime}(z_j)}{2f^{\prime}(z_j)(1-z_j)} \right)=0,
\end{multline*}
where $z_j=e^{\imath\theta_j}$. If  $p=\alpha$ and $2\imath q=\beta$,  this last equation gives 
\begin{equation}\label{deltae}
\displaystyle\frac{\partial E}{\partial\theta_j}
=
\Im\left(\frac{z_j(1-z_j)f^{\prime\prime}(z_j)-(n+\alpha-1-\beta-(n-\alpha-\beta-1)z_j)f^{\prime}(z_j)}{f^{\prime}(z_j)(1-z_j)} \right)=0.
\end{equation}

Let us write (\ref{deltae}) as
\begin{equation}\label{deltaeaux}
\displaystyle\frac{\partial E}{\partial\theta_j}
=\frac{1}{|g_n(z_j)|^2}\Im\left(\Pi_n(z_j)\overline{g_n(z_j)}\right)=0,
\end{equation}
where $\Pi_n(z)=z(1-z)f^{\prime\prime}(z)-(n+\alpha-1-\beta-(n-\alpha-\beta-1)z)f^{\prime}(z)$ and $g_n(z)=f^{\prime}(z)(1-z)$.

 Notice that if  $\Pi_n(e^{\imath\theta})=\kappa_n f(e^{\imath\theta})$, for some adequate constant $\kappa_n\in\co$, then  \eqref{deltaeaux} also holds. By comparing the coefficient $z^n$ in
$$
\kappa_n f(z)=z(1-z)f^{\prime\prime}(z)-(n+\alpha-1-\beta-(n-\alpha-\beta-1)z)f^{\prime}(z),
$$
one can deduce that $\kappa_n=-n(\alpha+\beta)$ and from the fact that the hypergeometric differential equation
$$
z(1-z)y^{\prime\prime}-(\alpha+n-\beta-1-(n-\alpha-\beta-1)z)y^{\prime}+n(\alpha+\beta) y=0,
$$
has a unique monic polynomial solution
$$
B_n\left(z;\frac{(\alpha-\beta)_{n+1}}{(\alpha+\beta)_{n+1}}\right)=
\displaystyle\frac{(2\alpha)_n}{(\alpha+\beta)_n}\, {_2F_1}(-n, \alpha+\beta; 2\alpha; 1-z),
$$
it can be deduced that if $\{e^{\imath\theta^{*}_1},\ldots,e^{\imath\theta^{*}_n}\}$ are the zeros of the para--orthogonal polynomial $B_n$, then  $\nabla_{\theta}E(\theta_1^{*},\ldots,\theta_n^{*})=0$.

 Consider now $p>0$, and the need to prove  that  the  energy function $E$ attains its  global minimum at the  point  $(\theta^{*}_1,\ldots,\theta^{*}_n)$.  Define the set $\Theta_0=\{(\theta_1,\ldots,\theta_n)\in [0,2\pi]^n:  \theta_1<\ldots<\theta_j<\ldots<\theta_n\}$. Notice that as $(\theta_1,\ldots,\theta_n)$ approaches to the boundary of the set  $\Theta_0$,  $E\rightarrow +\infty$, therefore the solution set $\Theta_1$ of the problem
\begin{equation}\label{probmin}
\min_{(\theta_1,\ldots,\theta_n)\in \Theta_0} E(\theta_1,\ldots,\theta_n),
\end{equation}
belongs to the interior of the set $\Theta_0$. It follows from the theory of constrained optimization \cite[pp. 327--328]{noc99} that the first--order optimality conditions (or the Karush--Kuhn--Tucker conditions, more precisely) reduce to the equation $\nabla_{\theta}E=0$.

In order to see if the local extremum $(\theta_1^{*},\ldots,\theta_n^{*})$ is a global minimum, one can check that the Hessian matrix  is  positive definite in the interior of $\Theta_0$. Indeed, a straightforward calculation shows that the Hessian matrix $H$
\begin{eqnarray*}
H=(h_{j,k}), \quad h_{j,k}=\frac{\partial^{2} E}{\partial \theta_j\partial \theta_k } ,
\end{eqnarray*}
reduces to
\begin{eqnarray*}
\dsty  h_{j,k}=&\left\{\begin{array}{rl}
          \dsty  \frac{1}{2}\left( \frac{p}{1-\cos \theta_k}+ \sum_{\substack{
i=1\\
 i\neq k}}^{n}\frac{1}{1-\cos(\theta_k-\theta_i)}\right), &   j=k ,\\
             \dsty  -\frac{1}{2}\left(\frac{1}{1-\cos(\theta_k-\theta_j)}\right),&  j\neq k.
            \end{array} \right.
\end{eqnarray*}

The above calculation shows that $H$ is real, symmetric, strictly diagonally dominant (see please \cite[Def. 6.1.9]{HJ}), and its diagonal entries are positive, from which it can be deduced   that $H$ is positive definite in the interior of $\Theta_0$, cf. \cite[Th. 6.1.10 (c)]{HJ},  therefore, the local extremum is a global minimum, provided that $p>0$.

 If one assumes $p<0$, evidently there is no global minimum for the energy function and this proves the theorem.
\end{proof}

   This section ends with some final remarks.  The sign of $\alpha$  strongly influences  the sign of the Hessian matrix. Considering  $-\frac{1}{2}<\alpha<0$,    the zeros of the para--orthogonal polynomial define a stationary point for the energy function. In this case the stationary point is  a saddle point for the  Hessian and, as could be seen in the proof of the theorem,   the energy function remains unbounded from below so that there is no global minimizer. This  is the case of a   negative electric charge fixed at $+1$ and $n$ free positive  charges varying at the boundary of the unit circle. This particular case   is interesting since the zeros of the para--orthogonal polynomials define a stationary point of the energy function. However they do not define a global minimum of the energy as in $\alpha>0$.

Recently, in \cite{simanek}, the author  provided a constructive method to find an electrostatic  model in some sense  for  zeros of para--orthogonal polynomials, by finding a  differential equation with rational coefficients  for the para--orthogonal polynomials. In the  case presented here,  the a prior knowledge that a  hypergeometric differential equation  already existed is used, which makes the model very similar to the classical one given by Stieltjes  for the Jacobi polynomials. 

 It would be an interesting problem to determine if the stationary point defined by the zeros of the $n$--th para--orthogonal polynomial for the  case $-\frac{1}{2}<\alpha<0$  is a local minimum, in other words, if one puts $n\geq 2$ unit masses    at the zeros of the $n$--th para--orthogonal polynomial associated to the measure $\omega(\theta;\alpha,\beta)$,   and one additional   mass point with value $-\frac{1}{2}<\alpha<0$  fixed at $+1$ and lets the system interact with the energy function defined by \eqref{E}, what can be said if a small perturbation over the masses is done?

\noindent\textit{Acknowledgements.}

 We thank the  anonymous reviewer whose comments and suggestions helped improve and clarify this manuscript. 
  
  The first author  was supported by FAPESP of Brazil (grant 2012/21042-0) and  the financial support of Ministerio de Econom\'{\i}a y Competitividad of Spain  (grant MTM2012--36732--C03-01). 
 The  second author  acknowledges the support of the Funda\c{c}\~ao de Amparo \`a Pesquisa do Estado de Minas Gerais--FAPEMIG (grant PPM--00478--15) and the Pr\'o-Reitoria de Pesquisa e P\'os-Gradua\c{c}\~ao da Universidade Federal de Uberl\^andia--PROPP/UFU.


\begin{thebibliography}{}

\bibitem{A66}  L.V. Ahlfors,  \textit{Complex analysis}, McGraw--Hill Book Company, 3th ed., 1953.

 

\bibitem{A82}  R. Askey,  Discussion of Szeg\"o's paper\emph{ ``Beitr\"age zur Theorie der Toeplitzschen Formen".}
In: G. Szeg\"o, Collected Works, Vol. I (R. Askey, ed.). Boston: Birkh\"auser, (1982)  303--305.

\bibitem{A85} R. Askey, Some open problems about special functions and computations, Rend. Sem. Mat. Univ. Politec. Torino, Special Volume  (1985) 1--22.

\bibitem{B78} E. Basor, Asymptotic formulas for Toeplitz determinants, Trans. Amer. Math. Soc. \textbf{239} (1978) 33--65.


\bibitem{gru98}   F.A. Gr\"unbaum, Variations on a theme of Heine and Stieltjes: an electrostatic interpretation of the zeros of certain polynomials,  J. Comp. Appl. Math. {\bf 99} (1) (1998) 189--194.

\bibitem{HJ} R. A. Horn and C. R. Johnson, {\em Matrix Analysis}, Cambridge University Press, New York,  2013.


\bibitem{ismail00} M.E. Ismail,  More on electrostatic models for zeros of orthogonal polynomials, Numerical functional analysis and optimization {\bf 21} (1--2)  (2000) 191--204.

\bibitem{JNT}  W.B. Jones, O. Nj{\aa}stad,  and W. J. Thron, {Moment theory, orthogonal polynomials, quadrature, and continued fractions associated with the unit circle}, Bull. London Math. Soc. {\bf 21} (1989) 113--152.


\bibitem{Lu}   Y.L.  Luke,  {\em The Special Functions and their approximations},  Academic Press, New York, Vol. I, 1969.



\bibitem{ranga} A. Sri Ranga, Szeg\"o polynomials from hypergeometric functions, Proc. Amer. Math. Soc.  {\bf 138} (12) (2010) 4259--4270.


\bibitem{simanek}  B. Simanek,  An electrostatic interpretation of the zeros of paraorthogonal polynomials on the unit circle,  preprint available at  arXiv:1501.05672v1


\bibitem{N86} N.M. Temme, {Uniform asymptotic expansion for a class of polynomials biorthogonal on the unit circle}, Constr. Approx. {\bf 2} (1986) 369--376.

\bibitem{noc99} S. J. Wright and J. Nocedal,  {\em Numerical optimization}, New York: Springer, 1999.





\end{thebibliography}
\end{document}